\documentclass[a4paper,12pt]{article}

\usepackage{amssymb}
\usepackage{amsthm}
\usepackage{amsmath}
\usepackage{graphicx}
\usepackage{array}
\usepackage{hhline}

\newtheorem{lemma}{Lemma}[section]
\newtheorem{theorem}[lemma]{Theorem}
\newtheorem{proposition}[lemma]{Proposition}

\theoremstyle{definition}
\newtheorem{definition}[lemma]{Definition}

\numberwithin{equation}{section}
\numberwithin{figure}{section}

\newcommand{\R}{\mathbb{R}}

\newcommand{\Z}{\mathbb{Z}}

\newcommand{\X}{\mathcal{X}}

\newcommand{\ie}{\emph{i.e.}}

\newcommand{\abs}[1]{\lvert#1\rvert}

\begin{document}

\title{Formula for Hermite multivariate interpolation and partial fraction decomposition}

\author{Hakop Hakopian \\
  \small Yerevan State University, Yerevan, Armenia \\
  \small Institute of Mathematics, Institute of Mathematics of NAS RA
}

\date{}
\maketitle

\begin{abstract}
We present a new formula for the Hermite multivariate interpolation problem in the framework of the Chung--Yao approach. By using the respective univariate interpolation formula, we obtain a direct and explicit solution to the classical partial fraction decomposition problem for rational functions, including the real case.
\end{abstract}



\section{Introduction}

\subsection{Chung--Yao Lagrange interpolation}

For $x = (x_1,\dots,x_k)$, $y = (y_1,\dots,y_k) \in \R^k$ and multi-index $\alpha = (\alpha_1,\dots,\alpha_k) \in \Z_{\ge 0}^k$, we adopt the standard multi-index notation:
\[
x \cdot y = \sum_{i=1}^k x_i y_i, \qquad
x^\alpha = \prod_{i=1}^k x_i^{\alpha_i}, \qquad
\abs{\alpha} = \sum_{i=1}^k \alpha_i, \qquad
\alpha! = \prod_{i=1}^k \alpha_i!.
\]

The space of polynomials of total degree at most $n$ in $k$ variables is
\[
\Pi_n^k = \Bigl\{ \sum_{\abs{\alpha} \le n} c_\alpha \, x^\alpha \Bigr\},
\qquad \dim\Pi_n^k = \binom{n+k}{k} =: N.
\]

Let $\mathcal L_m = \{L_1,\dots,L_m\}$ be a collection of $(k-1)$-dimensional hyperplanes in $\R^k$.

 Denote by $\mathbb{I}_k^m$ the set of all strictly increasing $k$-tuples from $\{1,\dots,m\}$:
\[
\alpha = (\alpha_1,\dots,\alpha_k) \in \mathbb{I}_k^m \iff 1 \le \alpha_1 < \cdots < \alpha_k \le m.
\]

\begin{definition}
The family $\mathcal L_m$ is in \emph{general position} if
\begin{enumerate}
\item the intersection of any $k$ distinct hyperplanes is a single point,
\item the intersection of any $k+1$ distinct hyperplanes is empty.
\end{enumerate}
If only condition (i) holds, we say $\mathcal L_m$ is \emph{admissible}.
\end{definition}

The intersection points are denoted
\[
x_\alpha := L_{\alpha_1} \cap \cdots \cap L_{\alpha_k}, \qquad \alpha \in \mathbb{I}_k^m.
\]

Note that condition (ii) means that all points $x_\alpha$ are distinct.

Assume now that $\mathcal L_{n+k} = \{L_1,\dots,L_{n+k}\}$ is in general position. Then there are exactly $N = \binom{n+k}{k}$ distinct intersection points. To simplify notation, we assume that the hyperplane $L_i,$ is given by a linear equation $L_i(x)=0,$ \ie  \ $L_i\in\Pi_1^k.$

\begin{theorem}[Chung--Yao \cite{ChungYao}]\label{thm:ChungYao}
For any data $\{c_\alpha : \alpha \in \mathbb{I}_k^{n+k}\}$ there exists a unique $p \in \Pi_n^k$ such that
\begin{equation}\label{eq:ChY}
p(x_\alpha) = c_\alpha \quad \forall\, \alpha \in \mathbb{I}_k^{n+k}.
\end{equation}
\end{theorem}

Note that the fundamental polynomial of  $x_\alpha$ is
\[
p_\alpha^\star(x) = \frac{1}{A_\alpha} \prod_{\substack{i=1 \\ i \notin \alpha}}^{n+k} L_i(x),
\]
where $A_\alpha$ is the normalizing constant so that $p_\alpha^\star(x_\alpha) = 1$. 

Then the Lagrange formula gives the polynomial satisfying \eqref{eq:ChY}:
\[
p(x) = \sum_{\alpha \in \mathbb{I}_k^{n+k}} c_\alpha \, p_\alpha^\star(x).
\]

\subsection{Hermite interpolation}

Now assume that $\mathcal L_{n+k}$ is admissible only. Let
\[
\X = \{x^{(1)},\dots,x^{(s)}\}
\]
be the set of all distinct intersection points of the hyperplanes of $\mathcal L_{n+k}$. 

Define \emph{multiplicity} of $x^{(i)}$ as
\[
m_i = \#\{j : x^{(i)} \in L_j,\ 1 \le j \le n+k\} - k + 1.
\]
Denote for $\alpha=(\alpha_1,\ldots,\alpha_k)\in \mathbb Z_{\ge 0}^k$ 
$$D^\alpha f= {{\partial^{|\alpha|}}\over{\partial {x_1}^{\alpha_1}\cdots\partial {x_k}^{\alpha_k}}}f.$$

The Hermite interpolation data consist of all partial derivatives up to total order $m_i-1$ at each point $x^{(i)}$. 

We say a point  $x^{(i)}$ is \emph{simple} if its multiplicity equals to $1$. Note that at simple points only the value of a polynomial is interpolated.

As it turns out (see \cite{Hak3}) the number of interpolation conditions in the case of  admissible set of hyperplanes $\mathcal L_{n+k}$ again equals to $N$ and the corresponding Hermite multivariate interpolation problem is unisolvent.

Below we present the Hermite multivariate  polynomial ionterpolation in the framework of the Chung--Yao approach.

\begin{theorem}[\cite{Hak3}]\label{thm:Hermite}
For any data $\{c_i^\alpha : 1 \le i \le s,\ \abs{\alpha} \le m_i-1\}$ there exists a unique $p \in \Pi_n^k$ satisfying
\begin{equation}\label{eq:Hermite}
D^\alpha p(x^{(i)}) = c_i^\alpha \quad \forall\, 1 \le i \le s,\ \forall\, \abs{\alpha} \le m_i-1.
\end{equation}
\end{theorem}

Next we discuss the problem of finding the polynomial satisfying the conditions \eqref{eq:Hermite}.

\subsection{New Hermite multivariate interpolation formula}

Let $f$ be sufficiently smooth. The Taylor polynomial of total degree $m$ for $f$ at $c \in \R^k$ is
\[
\mathcal{T}_{f,c,m}(x) = \sum_{\abs{\alpha} \le m} \frac{D^\alpha f(c)}{\alpha!} (x-c)^\alpha.
\]

It satisfies
\begin{equation}\label{eq:Taylor}
D^\alpha \mathcal{T}_{f,c,m}(c) = D^\alpha f(c) \quad \forall\, \abs{\alpha} \le m.
\end{equation}

Define the global vanishing polynomial
\[
\phi(x) = \prod_{j=1}^{n+k} L_j(x)
\]
and the corresponding polynomial vanishing outside the point $x^{(i)}$
\[
\phi_i(x) = \prod_{\substack{j=1 \\ x^{(i)} \notin L_j}}^{n+k} L_j(x).
\]

Let $p_f \in \Pi_n^k$ be the unique Hermite interpolant of $f$, \ie
\[
D^\alpha p_f(x^{(i)}) = D^\alpha f(x^{(i)}) \quad \forall\, 1 \le i \le s,\ \forall\, \abs{\alpha} \le m_i-1.
\]

\begin{proposition}
Let $\mathcal L_{n+k}$ be admissible. Then the following explicit  formula holds:
\[        
p_f(x) = \sum_{i=1}^s \phi_i(x) \cdot \mathcal{T}_{f/\phi_i,\, x^{(i)},\, m_i-1}(x).
\]
\end{proposition}
Let us call this Lagrange--Taylor formula.
 
\begin{proof}
It suffices to show that each fixed term
\[
p_i(x) := \phi_i(x) \cdot\mathcal T_i, \ \hbox{where}\ \mathcal T_i:= \mathcal{T}_{f/\phi_i,\, x^{(i)},\, m_i-1}(x)
\]
satisfies the following two groups of conditions:

\textbf{1. Vanishing at other points $x^{(r)}$, $r \neq i,$ up to total order $m_r-1$:}  
Indeed, exactly $m_r+k-1$ hyperplanes from $\mathcal L$ pass through $x^{(r)}$. At most $k-1$ of them can also pass through $x^{(i)}$ (since otherwise $x^{(r)} = x^{(i)}$). Therefore  at least $m_r$ linear factors of $\phi_i$ vanish at $x^{(r)}$. So all derivatives of $p_i$ up to order $m_r-1$ vanish at $x^{(r)}$.

\textbf{2. Correct reproduction at $x^{(i)}$ up to total order $m_i-1$:}  

By the multivariate Leibniz rule,
 \begin{equation}\label{eq:ma}
D^\alpha p_i(x^{(i)}) = \sum_{\beta \le \alpha} \binom{\alpha}{\beta} \bigl(D^\beta \phi_i\bigr)(x^{(i)}) \cdot \bigl(D^{\alpha-\beta} \mathcal{T}_i\bigr)(x^{(i)}).
\end{equation}

Since $\abs{\alpha} \le m_i-1$ and $\mathcal{T}_i$ reproduces all derivatives of $f/\phi_i$ up to order $m_i-1$ at $x^{(i)}$, we have
\[
D^{\alpha-\beta} \mathcal{T}_i(x^{(i)}) = D^{\alpha-\beta} \Bigl( \frac{f}{\phi_i} \Bigr)(x^{(i)})
\]
for every term in the sum \eqref{eq:ma}. Therefore
\[
D^\alpha p_i(x^{(i)}) = D^\alpha \Bigl( \phi_i \cdot \frac{f}{\phi_i} \Bigr)(x^{(i)}) = D^\alpha f(x^{(i)}).
\]

This completes the proof.
\end{proof}

In the next section we discuss an application of the Lagrange--Taylor formula to the univariate case.

\section{An application: partial fraction decomposition} 
\subsection{Preliminaries and the simple roots case}
Let $\pi$ and $\pi_n$ denote the spaces of all univariate polynomials and univariate polynomials of degree at most $n$, respectively.

Consider a rational function
$$R(x)={p(x)\over q(x)}, \ \hbox{where}\ p,q\in\pi,\ q\neq 0.$$
Here we assume that $\deg p=m,\ \deg q=n+1.$
 
$R$ is called an improper or proper rational function if $m\ge n+1$ or $m\le n,$ respectively. 

Any improper rational function can be decomposed as a sum of a polynomial and a proper rational function:

\begin{equation} \label{pqrsb}
{{p(x)}\over {q(x)}} = s(x) + {{r(x)}\over{q(x)}},
\end{equation}
where $s \in \pi$, $\deg s = m - (n+1)$, and $r \in \pi_n$ is the remainder of polynomial division:
\begin{equation}\label{pqrs}p(x)=s(x) q(x)+ r(x).
\end{equation}
First consider the well-known simplest case: the roots of $q$ are distinct complex numbers: $x_0, x_1,\ldots, x_n.$  Assume without loss of generality that $q$ is monic, so it can be factored as
$$q(x)=(x-x_0)\cdots (x-x_n).$$
Note that
\begin{equation}\label{qAx}q'(x_i)=\prod_{\substack{j=0 \\ j \neq i}}^n (x_i-x_j).\end{equation}

Clearly we obtain from \eqref{pqrs} that $r(x_i) = p(x_i)$ for each $i=0,\ldots,n.$

By the Lagrange formula,
\[
r(x) = \sum_{i=0}^n p(x_i) \prod_{\substack{j=0 \\ j \neq i}}^n \frac{x - x_j}{x_i - x_j}.
\]
Dividing by $q(x)$ and using \eqref{qAx} yields the well-known partial fraction form for distinct roots:
\[
\frac{r(x)}{q(x)} = \sum_{i=0}^n \frac{p(x_i)}{q'(x_i)} \frac{1}{x - x_i},
\]
and therefore
\[
\frac{p(x)}{q(x)} = s(x) + \sum_{i=0}^n \frac{c_i}{x - x_i}, \qquad c_i = \frac{p(x_i)}{q'(x_i)}.
\]

\subsection{{Two univariate interpolation formulas}}
Now assume that the roots of the denominator $q$ are multiple presented in the following form:
\begin{equation} \label{mset}\{t_0,\ldots,t_n\}=\{\underbrace{d_1,\ldots,d_1}_\text{$m_1$};\ldots; \underbrace{d_s,\ldots, d_s}_\text{$m_s$}\}, \end{equation}
where $D=\{d_1,\ldots,d_s\}$ is the set of distinct roots and $m=\{m_1,\ldots,m_s\}$ is the set of multiplicities,
$m_1+\cdots+m_s=n+1.$

Thus  $q$ has the following expansion:
$$q(x)=(x-t_0)\cdots (x-t_n)=(x-d_1)^{m_1}\cdots (x-d_s)^{m_s}.$$
The classical partial fraction decomposition theorem states that there exist a polynomial $s$ and constants $c_{ij}$ ($i=1,\dots,s$, $j=1,\dots,m_i$) such that
\[
\frac{p(x)}{q(x)} = s(x) + \sum_{i=1}^s \sum_{j=1}^{m_i} \frac{c_{ij}}{(x - d_i)^j}.
\]

To determine the constants $c_{ij}$, there is a long procedure, for example by reducing the problem to solving a linear system.

We will present below a direct and explicit solution to this problem, using the univariate analogue of the Lagrange-Taylor formula.

To also study the case of real rational functions, we will start with another similar univariate formula.
To this end, let us begin with the case of distinct roots and their distribution into different groups:
\begin{equation} \label{mset0}\{t_0,\ldots, t_n\}=\{\underbrace{d_{10},\ldots,d_{1m_1-1}}_\text{the first group};\ldots; \underbrace{d_{s0},\ldots, d_{sm_s-1}}_\text{the $s$-th group}\}. \end{equation}
Indeed, we have that
$$m_1+\cdots+m_s=n+1.$$
Denote
$$\psi(x):=\prod_{i=1}^s\prod_{j=0}^{m_i-1}(x-d_{ij})=\prod_{i=1}^s(x-d_{i0})\cdots(x-d_{im_i-1}),$$
$$\psi_j(x):={{\psi(x)}\over {(x-d_{j 0})\cdots(x-d_{j m_j-1})}}=\prod_{\substack{i=1 \\ i \neq j}}^s(x-d_{i0})\cdots(x-d_{im_i-1}).$$
Note that $\psi_j(d_{jk)}\neq 0\ \forall k=0,\ldots,m_j-1.$

The following grouped Lagrange interpolation formula holds:
\begin{equation}\label{eq:LagNewton}
\mathcal{P}_{f; t_0,\dots,t_n}(x) = \sum_{i=1}^s \psi_i(x) \cdot \mathcal{P}_{f/\psi_i;\ d_{i0},\dots,d_{i,m_i-1}}(x),
\end{equation}
where $\mathcal{P}$ denotes the unique interpolating polynomial of degree $\le n$.

Indeed, the right-hand side is a polynomial of degree at most $n$ (since $\deg \psi_i = n+1 - m_i$ and the local interpolant has degree $\le m_i-1$), and it matches $f\over{\psi_i}$ at every node $d_{ij}$ ($j=0,\dots,m_i-1$).

When all $m_i = 1$, formula \eqref{eq:LagNewton} reduces to the classical Lagrange formula.

Expressing the local  interpolants of  \eqref{eq:LagNewton} via the Newton form we obtain:
\begin{equation} \label{L-N}{\mathcal P}_{f,t_0,\ldots,t_n}=\sum_{i=1}^s \psi_i(x)\sum_{j=0}^{m_i-1}(x-d_{i0})\cdots(x-d_{ij-1})[d_{i0},\cdots, d_{ij}]{f\over {\psi_i}}.
\end{equation} 
This formula will be used in the real decomposition case.

Note that coalescing the nodes in each group ($d_{ij} \to d_i$) transforms \eqref{L-N} into the univariate Lagrange--Taylor formula:
\begin{equation}\label{eq:LagTaylor}
\mathcal{P}_{f;\ t_0,\dots,t_n}(x) = \sum_{i=1}^s q_i(x) \sum_{j=0}^{m_i-1} \frac{1}{j!} \left( \frac{f}{q_i} \right)^{(j)}(d_i) (x - d_i)^j,
\end{equation}
where
$$q_i(x)={q(x)\over (x-d_i)^{m_i}},\ \  q(x)=(x-d_1)^{m_1}\cdots (x-d_k)^{m_k}.$$

\subsection{The decomposition of rational functions in the general case}

Applying \eqref{eq:LagTaylor} to the remainder polynomial $f=r \in \pi_n$ (which interpolates itself) gives
\[
r(x) = \sum_{i=1}^s q_i(x) \sum_{j=0}^{m_i-1} \frac{1}{j!} \left( \frac{r}{q_i} \right)^{(j)}(d_i) (x - d_i)^j.
\]
Dividing by $q(x)$ yields
\[
\frac{r(x)}{q(x)} = \sum_{i=1}^s \sum_{j=0}^{m_i-1} \frac{1}{j!} \left( \frac{r}{q_i} \right)^{(j)}(d_i) \frac{1}{(x - d_i)^{m_i-j}}.
\]
Combining with the polynomial part $s(x)$ we obtain the explicit partial fraction decomposition:
\[
\frac{p(x)}{q(x)} = s(x) + \sum_{i=1}^s \sum_{j=0}^{m_i-1} \frac{c_{ij}}{(x - d_i)^{m_i-j}},
\]
where
\[
c_{ij} = \frac{1}{j!} \left( \frac{r}{q_i} \right)^{(j)}(d_i) = \frac{1}{j!} \left( \frac{p}{q_i} \right)^{(j)}(d_i).
\]
The last equality above follows from the relation ${\mathcal P}_{p;t_0,\ldots,t_n}=r,$ which in turn follows from
\eqref{pqrs}.

\subsection{Real partial fraction decomposition}

Now assume $p$ and $q$ are real polynomials. The non-real roots of $q$ appear in conjugate pairs with equal multiplicities. Let $\{a_1,\dots,a_s\}$ be the real roots with multiplicities $m_1,\dots,m_s$, and let $\{b_\nu = c_\nu + i d_\nu,\ \bar{b}_\nu = c_\nu - i d_\nu\}$ ($d_\nu > 0$) be the complex conjugate pairs with multiplicities $\mu_\nu$ each ($\nu=1,\dots,\sigma$):

\begin{equation} \label{LagTk}\{t_0,\ldots,t_n\}=\{\underbrace{a_1}_\text{$m_1$},\ldots,\underbrace{a_s}_\text{$m_s$}, \underbrace{b_1,\bar b_1}_\text{$\mu_1$},\ldots, \underbrace{b_\sigma,\bar b_\sigma}_\text{$\mu_\sigma$}\}.\end{equation}

Then
\[
\sum_{\nu=1}^s m_\nu + 2 \sum_{\nu=1}^\sigma \mu_\nu = n+1,
\]
and
\[
q(x) = \prod_{\nu=1}^s (x - a_\nu)^{m_\nu} \prod_{\nu=1}^\sigma (x^2 + u_\nu x + v_\nu)^{\mu_\nu},
\]
where $u_\nu = -2c_\nu$, $v_\nu = c_\nu^2 + d_\nu^2$.

Define
\begin{equation}\label{psinu}
\psi_\nu(x) = \frac{q(x)}{(x - a_\nu)^{m_\nu}}, \qquad \nu=1,\dots,s,
\end{equation}
\begin{equation}\label{etanu}
\eta_\nu(x) = \frac{q(x)}{(x^2 + u_\nu x + v_\nu)^{\mu_\nu}}, \qquad \nu=1,\dots,\sigma,
\end{equation}
\[\hbox{where}\qquad q(x):=\prod_{\nu =1}^s{(x-a_\nu )}^{m_\nu} \prod_{\nu=1}^\sigma (x^2+u_\nu x+v_\nu)^{\mu_\nu}.\]

Now, by using the formulas \eqref{eq:LagTaylor}  and \eqref{L-N} where the knots are grouped as in \eqref{LagTk}, we get
$${\mathcal P}_{f,t_0,\ldots,t_n}=S_1+S_2,$$
where
$$S_1(x)=\sum_{\nu =1}^s \psi_\nu (x)\sum_{k=0}^{m_\nu -1}{1\over k!}{\left(f\over \psi_\nu \right)}^{(k)}(a_\nu) (x-a_\nu)^k,
$$ and
$$S_2(x)=\sum_{\nu =1}^\sigma \eta_\nu(x)\sum_{k=0}^{\mu_\nu -1}(x^2+u_\nu x+v_\nu)^k[\underbrace{b_\nu, \bar b_\nu,\ldots,b_\nu,\bar b_\nu}_\text{$2k$},b_i]{f\over {\eta_\nu}}
$$
$$+\sum_{\nu=1}^\sigma \eta_\nu(x)\sum_{k=0}^{\mu_\nu-1}(x^2+u_\nu x+v_\nu)^k(x-b_\nu)[\underbrace{b_\nu,\bar b_\nu,\ldots,b_\nu,\bar b_\nu}_\text{$2k+2$}]{f\over {\eta_\nu}}
$$
$$=\sum_{\nu=1}^\sigma \eta_\nu(x)\sum_{k=0}^{\mu_\nu-1}(M_{\nu k}x+N_{\nu k})(x^2+u_\nu x+v_\nu)^k,
$$
where
$$M_{\nu k}x+N_{\nu k}=[\underbrace{b_\nu,\bar b_\nu,\ldots,b_\nu,\bar b_\nu}_\text{$2k$},b_\nu]{f\over {\eta_\nu}}+(x-b_\nu)[\underbrace{b_\nu,\bar b_\nu,\ldots,b_\nu,\bar b_\nu}_\text{$2k+2$}]{f\over {\eta_\nu}}.$$
By equating here the coefficients of $x$ and free terms, we get
\begin{equation}\label{Mij} M_{\nu k}=[\underbrace{b_\nu,\bar b_\nu,\ldots,b_\nu,\bar b_\nu}_\text{$2k+2$}]{f\over {\eta_\nu}},\end{equation}
$$N_{\nu k}=[\underbrace{b_\nu,\bar b_\nu,\ldots,b_\nu,\bar b_\nu}_\text{$2k$},b_\nu]{f\over {\eta_\nu}}-b_\nu[\underbrace{b_\nu,\bar b_\nu,\ldots,b_\nu,\bar b_\nu}_\text{$2k+2$}]{f\over {\eta_\nu}}$$
$$=[\underbrace{b_\nu,\bar b_\nu,\ldots,b_\nu,\bar b_\nu}_\text{$2k+2$}]\left\{(t-\bar b_\nu){f\over {\eta_\nu}}-b_\nu{f\over {\eta_\nu}}\right\}.$$
Above, in the last equality, we used the relation
$$[x_0,\ldots,x_n]\{(t-x_n)f(t)\}= [x_0,\ldots,x_{n-1}]f.$$
Therefore we have
\begin{equation}\label{Nij} N_{\nu k}=[\underbrace{b_\nu,\bar b_\nu,\ldots,b_\nu,\bar b_\nu}_\text{$2k+2$}] \left\{(t-2c_\nu){f\over {\eta_\nu}}\right\}.\end{equation}
Thus we get the following formula:
$${\mathcal P}_{f,t_0,\ldots,t_n}=\sum_{\nu=1}^s \psi_\nu(x)\sum_{k=0}^{m_k-1}C_{\nu k} (x-a_\nu)^k+\sum_{\nu=1}^\sigma \eta_\nu(x)\sum_{k=0}^{\mu_\nu-1}(M_{\nu k}x+N_{\nu k})(x^2+u_\nu x+v_\nu)^k,$$
where $C_{\nu k}={1\over k!}{\left(f\over \phi_\nu\right)}^{(k)}(a_\nu)$, while  the numbers $M_{\nu k}$ and $N_{\nu k}$ are given in \eqref{Mij} and \eqref{Nij}, respectively.
The last two numbers are real, in view of the relation
$$ \xi=[x_0,\ldots,x_n]f \implies \bar\xi=[\bar x_0,\ldots, \bar x_n]\bar f.$$

Now let $f=r\in\pi_n.$ Then we have that ${\mathcal P}_{f,t_0,\ldots,t_n}=r.$

Therefore the above formula holds for any polynomial $r\in\pi_n$ in the following form:
\begin{equation}\label{rs1s2}r(x)=\sum_{\nu=1}^s \psi_\nu(x)\sum_{k=0}^{m_\nu-1}E_{\nu k} (x-a_\nu)^k+\sum_{\nu=1}^\sigma \eta_\nu(x)\sum_{k=0}^{\mu_\nu-1}(M_{\nu k}x+N_{\nu k})(x^2+u_\nu x+v_\nu)^k,\end{equation}
where $$E_{\nu k}={1\over k!}{\left(r\over \psi_\nu\right)}^{(k)}(a_\nu), \quad M_{\nu k}=[\underbrace{b_\nu,\bar b_\nu,\ldots,b_\nu,\bar b_\nu}_\text{$2k+2$}]{r\over {\eta_\nu}},$$ and 
$$ N_{\nu k}=[\underbrace{b_\nu,\bar b_\nu,\ldots,b_\nu,\bar b_\nu}_\text{$2k+2$}] \left\{(t-2 c_\nu){r\over {\eta_\nu}}\right\}.$$
Finally, by using \eqref{pqrsb} and dividing the both sides of \eqref{rs1s2} by $q(x)$, for ${r(x)\over q(x)},$ we get the following explicit formula for the  decomposition of real rational functions into real partial fractions:
\begin{equation}\label{irakan}{p(x)\over q(x)}=s(x)+\sum_{\nu=1}^s \sum_{k=0}^{m_\nu-1}{E_{\nu k}\over (x-a_\nu)^{n_\nu-k}}+\sum_{\nu=1}^\sigma \sum_{k=0}^{\mu_\nu-1}{(M_{\nu k}x+N_{\nu k})\over (x^2+u_\nu x+v_\nu)^{\mu_\nu-k}},\end{equation}
where
$$E_{\nu k}={1\over k!}{\left(p\over \psi_\nu\right)}^{(k)}(a_\nu), \quad M_{\nu k}=[\underbrace{b_\nu,\bar b_\nu,\ldots,b_\nu,\bar b_\nu}_\text{$2k+2$}]{p\over {\eta_\nu}},$$  
$$ N_{\nu k}=[\underbrace{b_\nu,\bar b_\nu,\ldots,b_\nu,\bar b_\nu}_\text{$2k+2$}] \left\{(t-2c_\nu){p\over {\eta_\nu}}\right\}.$$
The polynomials $\psi_\nu(x)$ and $\eta_\nu(x)$ here are given in \eqref{psinu} and \eqref{etanu}, respectively.

In the final expressions of $E_{\nu k}$, $M_{\nu k},$ and $N_{\nu k}$, compared to the previous ones, we have replaced the polynomial $r$ with $p$. The validity of this replacement follows from the relation ${\mathcal P}_{p;t_0,\ldots,t_n}=r,$ which itself follows from
\eqref{pqrs}.


\end{document}